\documentclass[12pt]{amsart}
\usepackage[dvips]{graphicx}
\usepackage{amsmath}
\usepackage{amssymb}
\usepackage{amscd}
\usepackage{mathrsfs}
\usepackage{mathtools}
\usepackage{color}
\usepackage{tikz}
\usetikzlibrary{cd}
\usetikzlibrary{intersections,calc,arrows}

\usepackage[all]{xy}%

\DeclareFontEncoding{OT2}{}{}
\DeclareFontSubstitution{OT2}{cmr}{m}{n}
\DeclareFontFamily{OT2}{cmr}{\hyphenchar\font45 }
\DeclareFontShape{OT2}{cmr}{m}{n}{%
   <5><6><7><8><9>gen*wncyr%
   <10><10.95><12><14.4><17.28><20.74><24.88>wncyr10}{}
\DeclareFontShape{OT2}{cmr}{b}{n}{%
   <5><6><7><8><9>gen*wncyb%
   <10><10.95><12><14.4><17.28><20.74><24.88>wncyb10}{}

\DeclareMathAlphabet{\mathcyr}{OT2}{cmr}{m}{n}
\DeclareMathAlphabet{\mathcyb}{OT2}{cmr}{b}{n}
\SetMathAlphabet{\mathcyr}{bold}{OT2}{cmr}{b}{n}

\DeclareMathOperator{\Ker}{Ker}

\theoremstyle{plain}
    \newtheorem{theorem}{Theorem}[section]
\theoremstyle{definition}
    \newtheorem{definition}[theorem]{Definition}
    \newtheorem{proposition}[theorem]{Proposition}
    
    \newtheorem{corollary}[theorem]{Corollary}
    
    \newtheorem{example}[theorem]{Example}
    \newtheorem{remark}[theorem]{Remark}

    \newtheorem{procedure}[theorem]{Procedure}

\newcommand{\bZ}{\mathbb{Z}}
\newcommand{\bQ}{\mathbb{Q}}
\newcommand{\bR}{\mathbb{R}}

\newcommand{\cA}{\mathcal{A}}

\newcommand{\cS}{\mathcal{S}}

\newcommand{\cZ}{\mathcal{Z}}

\newcommand{\frH}{\mathfrak{H}}

\newcommand{\sh}{\mathbin{\mathcyr{sh}}}

\newcommand{\bk}{\boldsymbol{k}}
\newcommand{\bl}{\boldsymbol{l}}
\newcommand{\bp}{\boldsymbol{p}}

\newcommand{\wcA}{\widehat{\mathcal{A}}}
\newcommand{\wcS}{\widehat{\mathcal{S}}}
\newcommand{\wphi}{\widehat{\phi}}
\newcommand{\wPhi}{\widehat{\Phi}}

\newcommand{\wt}{\mathrm{wt}}
\newcommand{\dep}{\mathrm{dep}}
\newcommand{\CRT}{\mathrm{CRT}}

\newcommand{\h}{\mathrm{h}}
\newcommand{\rt}{\mathrm{rt}}

\makeatletter
    
    \@addtoreset{equation}{section}
  \makeatother

\newcommand{\jump}[1]{\ensuremath{[\![#1]\!]} }

\address{Global Education Center, Waseda University, 1-6-1, Nishi-Waseda, Shinjuku-ku, Tokyo, 169-8050, Japan}
\email{m-ono@aoni.waseda.jp}

\title{$t$-adic symmetrization map on harmonic algebra}
\author{Masataka Ono}

\begin{document}

\maketitle

\begin{abstract}
Bachmann, Takeyama and Tasaka introduced a $\bQ$-linear map $\phi$, which we call the \emph{symmetrization map} in this paper, on the harmonic algebra $\frH^1$. They calculated $\phi(w)$ explicitly for an element $w$ in $\frH^1$ related to the multiple zeta values of Mordell--Tornheim type. In this paper, we introduce its $t$-adic generalization $\wphi$ and calculate $\wphi(w)$ for an element $w$ in $\frH^1\jump{t}$ constructed from the theory of $2$-colored rooted tree.
\end{abstract}

\section{Introduction}

\subsection{Notation}
For a tuple of non-negative integers $\bk=(k_1, \ldots, k_r)$, we set $\wt(\bk) \coloneqq k_1+\cdots+k_r$ and $\dep(\bk) \coloneqq r$, and we call them the \emph{weight} and the \emph{depth}, respectively. We call a tuple of non-negative integers an \emph{index} if all entries are positive. In particular, there is a unique index of depth 0, which we call the \emph{empty index} and denote by $\varnothing$. We say an index $\bk=(k_1, \ldots, k_r)$ \emph{admissible} if $\bk=\varnothing$ or $k_r\ge2$.

For non-empty tuples $\bk=(k_1, \ldots, k_r), \bl =(l_1, \ldots, l_r) \in \bZ^{r}_{\ge0}$, set $\overline{\bk} \coloneqq (k_r, \ldots, k_1)$, $\bk+\bl \coloneqq (k_1+l_2, \ldots, k_r+l_r), b\binom{\bk}{\bl} \coloneqq \prod_{i=1}^r\binom{k_i+l_i-1}{l_i}$, and $\binom{l-1}{l} \coloneqq \delta_{l, 0}$ for a non-negative integer $l$. Moreover, for $0 \le i \le r$, set $\bk_{[i]}\coloneqq(k_1, \ldots, k_i), \bk^{[i]} \coloneqq (k_{i+1}, \ldots, k_r)$, and we understand $\bk_{[0]}=\bk^{[r]}=\varnothing$.

\subsection{$t$-adic symmetric multiple zeta values}
Hereafter, let $M$ be a non-negative integer. For an index $\bk=(k_1, \ldots, k_r)$, set
\begin{equation*}
\zeta^{}_{M}(\bk)
\coloneqq
\sum_{0<n_1<\cdots<n_r<M}\frac{1}{n^{k_1}_1\cdots n^{k_r}_r}.
\end{equation*}
We understand that the empty sum is $0$ and $\zeta^{}_{M}(\varnothing)=1$. If $\bk$ is admissible, the limit $\lim_{M \rightarrow \infty}\zeta^{}_M(\bk)$ exists.
Let $\zeta(\bk)$ denote this limit and we call it the \emph{multiple zeta value} (MZV).

Many variants of MZV are known today. In this paper, we focus on \emph{$t$-adic symmetric multiple zeta values} ($\wcS$-MZV, \cite{HMO21}, \cite{OSY20}). 
For an index $\bk=(k_1, \ldots, k_r)$, set
\begin{equation*}
\zeta^{\sh}_{\wcS}(\bk)
\coloneqq
\sum_{i=0}^{r}(-1)^{\bk^{[i]}}\zeta^{\sh}(\bk_{[i]})
\sum_{\bl \in \bZ^{r-i}_{\ge0}}b\binom{\bk^{[i]}}{\bl}\zeta^{\sh}(\overline{\bk^{[i]}+\bl})t^{\wt(\bl)} \in \cZ\jump{t}.
\end{equation*}
Here, $\zeta^{\sh}(\bk)$ is the shuffle regularized multiple zeta value \cite{IKZ06} which is an element in the $\bQ$-algebra $\cZ \subset \bR$ generated by all MZVs. We set $\zeta^{\sh}_{\cS}(\bk) \coloneqq \zeta^{\sh}_{\wcS}(\bk)|_{t=0} \in \cZ$. 
We define the $\wcS$-MZV $\zeta^{}_{\wcS}(\bk)$ and the \emph{symmetric multiple zeta value} ($\cS$-MZV) $\zeta^{}_{\cS}(\bk)$ by
\begin{align*}
\zeta^{}_{\wcS}(\bk)
&\coloneqq \zeta^{\sh}_{\wcS}(\bk) \bmod{\zeta(2)} \in (\cZ/\zeta(2)\cZ)\jump{t}, \\
\zeta^{}_{\cS}(\bk)
&\coloneqq \zeta^{\sh}_{\cS}(\bk) \bmod{\zeta(2)} \in \cZ/\zeta(2)\cZ.
\end{align*}
Note that the notion $\wcS$-MZV appeared first in the literature in Jarossay's article \cite[Definition 1.3]{J21}. He defined it in terms of Drinfeld associator, and named it \emph{$\Lambda$-adjoint multiple zeta values}, where $\Lambda$ is an indeterminate corresponding to our $t$. It is conjectured that $\zeta^{}_{\cS}(\bk)$ (resp. $\zeta^{}_{\wcS}(\bk)$) and the finite multiple zeta values $\zeta^{}_{\cA}(\bk)$ (resp.~$\bp$-adic finite multiple zeta values $\zeta^{}_{\wcA}(\bk)$) satisfies the same $\bQ$-linear relations (resp. $\bQ$-linear or $\bp$-adic/$t$-adic relations). For details, see \cite[Conjecture 5.3.2]{J21}, \cite[Conjecture 9.5]{Kan19}, \cite{KZ21}, and \cite[Conjecture 4.3]{OSY20}.

\subsection{$t$-adic symmetrization map on the harmonic algebra}
Following \cite{H97}, we introduce the algebraic setup of MZVs and $\wcS$-MZVs. Let $\frH \coloneqq \bQ\langle x,y\rangle$ be the non-commutative polynomial ring over $\bQ$ with variables $x$ and $y$, and set $\frH^1 \coloneqq \bQ+y\frH \subset \frH$. Note that $\frH^1$ is generated by $z_k \coloneqq yx^{k-1} \ (k \in \bZ_{\ge1})$. Let $Z^{\sh} \colon \frH^1 \rightarrow \bR$ be a $\bQ$-linear map defined by $Z^{\sh}(z_{\bk})\coloneqq \zeta^{\sh}(\bk)$ for any index $\bk$. Here, we set $z_{\bk}=z_{k_1}\cdots z_{k_r}$ if $\bk=(k_1, \ldots, k_r)$  (we understand $z_{\varnothing}=1$). Similarly, we define a $\bQ$-linear map $Z^{\sh}_{\wcS} \colon \frH^1 \rightarrow \cZ\jump{t}$ by $Z^{\sh}_{\wcS}(z_{\bk}) \coloneqq \zeta^{\sh}_{\wcS}(\bk)$. 

Consider the $\bQ\jump{t}$-linear map $\wphi \colon \frH^1\jump{t} \rightarrow \frH^1\jump{t}$ defined by 
\begin{equation}\label{eq:sym_map_H1}
\wphi(z_{\bk})
\coloneqq \sum_{i=0}^{r}(-1)^{\wt(\bk^{[i]})}z_{\bk_{[i]}}\sh 
\sum_{\bl \in \bZ^{l-i}_{\ge0}}b\binom{\bk^{[i]}}{\bl}z_{\overline{\bk^{[i]}+\bl}}t^{\wt(\bl)}
\end{equation}
for an index $\bk$. Here, $\sh \colon \frH \times \frH \rightarrow \frH$ is the shuffle product on $\frH$ (see the Definition \ref{def:sh_prod}). Note that the $\bQ$-linear map $\phi \coloneqq \wphi|_{\frH^1} \colon \frH^1 \rightarrow \frH^1$ was introduced first in \cite[(4.6)]{BTT21}. Since $Z^{\sh}$ is homomorphism with respect to $\sh$, we have $\zeta^{\sh}_{\wcS}(\bk)=Z^{\sh} \circ \wphi (z_{\bk})$. We call $\wphi$ (resp. $\phi$) the \emph{$t$-adic symmetrization map on $\frH^1$} (resp. the \emph{symmetrization map on $\frH^1$}).

Bachmann, Takeyama and Tasaka \cite[Theorem 4.6]{BTT21} calculated the element $\phi\bigl((z_{k_1}\sh \cdots \sh z_{k_r})x^{k_{r+1}}\bigr)$ in $\frH^1$ for positive integers $k_1, \ldots, k_{r+1}$. 

\begin{theorem}\label{thm:BTT}
For positive integers $k_1, \ldots, k_{r+1}$, we have
\begin{equation}\label{eq:phi_MT}
\phi\bigl((z_{k_1} \sh \cdots \sh z_{k_r})x^{k_{r+1}}\bigr)
=\sum_{i=1}^{r+1}(-1)^{k_i+k_{r+1}}
(z_{k_1}\sh \cdots \sh \check{z}_{k_{i}} \sh \cdots \sh z_{k_{r+1}})x^{k_i}.
\end{equation}
Here the symbol $\check{z}_{k_i}$ means that the factor $z_{k_i}$ is skipped.
\end{theorem}
Since $Z^{\sh} \circ \phi\bigl((z_{k_1}\sh \cdots \sh z_{k_r})x^{k_{r+1}}\bigr)$ coincides with the $\cS$-MZV of Mordell--Tornheim type (MT-type) (\cite[Theorem 1.6]{BTT21}, \cite[Remark 4.9]{OSY20}), Theorem \ref{thm:BTT} gives a formula of $\cS$-MZV of MT-type in terms of MZV of MT-type (see also \cite[Proposition 4.8]{OSY20}).

\subsection{$2$-colored rooted trees and main theorem}
In our recent paper \cite{OSY20}, Seki, Yamamoto and the author defined $\wcS$-MZV associated with 2-colored rooted tree. The $2$-colored rooted tree is a tuple $X=(V, E, \rt, X)$ consisting of a finite tree $(V, E)$ with some additional data $(\rt, V_{\bullet})$, introduced first by the author in \cite[Definition 1.2]{O17} (for the precise definition, see Definition \ref{def:2crt}). We gave an algorithm to constructed from the element $w(X, \bk)$ in $\frH^1$ from a ``harvestable pair" $(X, \bk)$, that is, a tuple consisting of a $2$-colored rooted tree $X=(V, E, \rt, X)$ and a tuple $\bk \coloneqq (k_e)_{e \in E} \in \bZ^{E}_{\ge0}$ called an index on $X$ with certain conditions (see Definition \ref{def:harv_pair}). Note that the element $(z_{k_1}\sh \cdots \sh z_{k_r})x^{k_{r+1}}$ is constructed a certain harvestable pair $(X, \bk)$ \cite[Definition 4.5, Proposition 4.8]{OSY20}. In this paper, we calculate $\wphi(w)$ for any element $w=w(X, \bk) \in \frH^1$ constructed from some pair $(X, \bk)$ containing Theorem \ref{thm:BTT} as a special case. 

To state the main theorem, we give some notation. For a $2$-colored rooted tree $X=(V, E, \rt, V_{\bullet})$ with $\rt \in V_{\bullet}$ and an essentially positive (see Definition \ref{def:ess_pos}) index $\bk$ on $X$, let $(X_{\h}, \bk_{\h})$ denote its harvestable form constructed by the algorithm in \cite[Proposition 3.18]{OSY20}, and $w(X_{\h}, \bk_{\h})$ be the element in $\frH^1$ constructed by the algorithm in \cite[Theorem 3.17]{OSY20} from $(X_{\h}, \bk_{\h})$. For $v, v' \in V$, $P(v, v')$ denotes the path starting from $v$ to $v'$. For an index $\bk=(k_e)_{e\in E}$ on $X$ and $E' \subset E$, set $k_{E'} \coloneqq \sum_{e \in E'}k_{e}$.Finally, for subsets $E', E'' \subset E, E_0 \subset E' \cap E''$ and $\bk=(k_e)_{e \in E'} \in \bZ^{E'}_{\ge0}, \bl=(l_e)_{e \in E''} \in \bZ^{E''}_{\ge0}$, set $b_{E_0}\binom{\bk}{\bl} \coloneqq \prod_{e \in E_0} \binom{k_e+l_e-1}{l_e}$.

\begin{theorem}\label{mainthm:calu_phi}
Let $X=(V, E, \rt, V_{\bullet})$ be a 2-colored rooted tree with $\rt \in V_{\bullet}$ and $\bk \coloneqq (k_e)_{e \in E}$ an essentially positive index on $X$. Then we have the following equality in $\frH^1\jump{t}$.
\begin{equation*}
\wphi\bigl(w(X_{\h}, \bk_{\h})\bigr)
=\sum_{v \in V_{\bullet}}(-1)^{k_{P(\rt, v)}}
\sum_{\bl \in \bZ^{P(\rt, v)}_{\ge0}} b_{P(\rt, v)}\binom{\bk}{\bl}
w(X_{v, \h}, (\bk+\bl)_{\h})t^{l_{P(\rt, v)}}
\end{equation*}
Here, we set $X_v\coloneqq (V, E, v, V_{\bullet})$. In particular, we have
\begin{equation*}
\phi\bigl(w(X_{\h}, \bk_{\h})\bigr)
=\sum_{v \in V_{\bullet}}(-1)^{k_{P(\rt, v)}}w(X_{v, \h}, \bk_{\h}).
\end{equation*}
\end{theorem}
%

\subsection{Contents}
The contents of this paper is as follows. In Section $2$, we review the theory of $2$-colored rooted tree along with \cite{OSY20} which contains the definition of harvestable pair, and we introduce some product structures on the $\bQ$-vector space $2\CRT$ generated by $2$-colored rooted trees. In Section $3$, we introduce the $t$-adic symmetrization map $\wPhi$ on $2\CRT$, and prove Theorem \ref{mainthm:calu_phi}. We also give some identities in $\frH^1\jump{t}$ as an application of Theorem \ref{mainthm:calu_phi} which contains a generalization of Theorem \ref{thm:BTT} to $t$-adic case.

\section{The theory of 2-colored rooted trees}

In this section, we recall the theory of 2-colored rooted tree and the $\wcS$-MZV associated with them along with \cite{OSY20}, and we introduce a product structure on $2$-colored rooted trees.

\subsection{Notion of 2-colored rooted tree}

\begin{definition}[{\cite[Definition 1.2]{O17}, \cite[Definition 3.1]{OSY20}}]\label{def:2crt}
A 2-colored rooted tree is a tuple $X=(V, E, \rt, V_{\bullet})$ consisting of the following data.
\begin{enumerate}
\item $(V, E)$ is a finite tree with the set of vertices $V$ and the set of edges $E$. Note that $\#V=\#E+1<\infty$.
\item $\rt \in V$ is a vertex, called the \emph{root}.
\item $V_{\bullet}$ is a subset of $V$ containing all terminals of $(V, E)$. Here, a \emph{terminal vertex} means a vertex of degree $1$.
\end{enumerate}
\end{definition}

For a 2-colored rooted tree $X=(V, E, \rt, V_{\bullet})$ and an index $\bk=(k_e)_{e \in E}$, set
\begin{align*}
\zeta^{}_{M}(X; \bk)
\coloneqq
\sum_{\substack{(m_v)_{v \in V_{\bullet}} \in \bZ^{V_{\bullet}}_{\ge1} \\ \sum_{v \in V_{\bullet}}m_v=M}}
\prod_{e \in E}
\biggl(\sum_{v \in V_{\bullet} \text{ s.t. }e \in P(\rt, v)}m_v\biggr)^{-k_e} \in \bQ
\end{align*}
and 
\begin{align*}
\zeta^{}_{\wcS, M}(X; \bk)
&\coloneqq
\sum_{u \in V_{\bullet}}\zeta^{}_{M}(X, u; \bk), \\
\zeta^{}_{M}(X, u; \bk)
&\coloneqq
\sum_{(m_v)_{v \in V_{\bullet}} \in I_M(V_{\bullet}, u)}
\prod_{e \in E} \biggr(\sum_{v \in V_{\bullet} \text{ s.t. }e \in P(\rt, v)}(m_v+\delta_{u,v}t)\biggr)^{-k_e} \in \bQ\jump{t}
\ (u \in V_{\bullet}).
\end{align*} 
Here, for $u \in V_{\bullet}$, $\delta_{u,v}$ is the Kronecker's delta and 
\begin{align*}
I_M(V_{\bullet}, u)
\coloneqq
\left.\biggl\{
(m_v)_{v \in V_{\bullet}} \in \bZ^{V_{\bullet}}
\ \right| \
m_v>0 \ (v \neq u), -M<m_u<0, \sum_{v \in V_{\bullet}}m_v=0
\biggr\}.
\end{align*}

We use diagrams to indicate $2$-colored rooted trees. 
For $X=(V, E, \rt, V_\bullet)$, 
the symbol $\bullet$ (resp.~$\circ$) denotes a vertex in $V_{\bullet}$ 
(resp.~$V_\circ\coloneqq V\setminus V_\bullet$) which is not the root. 
We use the symbol $\blacksquare$ or $\square$ to denote the root 
according to whether the root belongs to $V_\bullet$ or not. 
If endpoints of an edge $e$ are $v$ and $v'$, then we will sometimes express $e$ by the set $\{v,v'\}$.

\begin{example}\label{ex:2crt_EZ}
For an integer $r\geq 0$, let us consider the linear tree with $r+1$ vertices $v_1,\ldots,v_{r+1}$ 
and $r$ edges $e_a\coloneqq\{v_a,v_{a+1}\}$ ($a=1,\ldots,r$). 
We define the $2$-colored rooted tree $X=(V, E, \rt, V_\bullet)$ 
by setting $\rt\coloneqq v_{r+1}$ and $V_\bullet\coloneqq V=\{v_1,\ldots,v_{r+1}\}$. 
We identify an index $\bk=(k_1,\ldots,k_r)$ in the usual sense 
with an index on $X$ by setting $k_a\coloneqq k_{e_a}$. 
This situation is indicated by the diagram: 
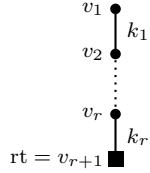
\begin{equation*}
\begin{tikzpicture}[thick]
\coordinate (v_1) at (0,2.0) node at (v_1) [left] {\tiny $v_1$};
\coordinate (v_2) at (0,1.4) node at (v_2) [left] {\tiny $v_2$};
\coordinate (v_r) at (0,0.6) node at (v_r) [left] {\tiny $v_r$};
\coordinate (v_{r+1}) at (0,0) node at (v_{r+1}) [left] {\tiny $\rt=v_{r+1}$};
\draw  (0,1.4) -- node [right] {\tiny $k_1$} (0,2.0);
\draw [dotted] (0,0.6) -- (0,1.4);
\draw  (0,0) -- node [right] {\tiny $k_r$} (0,0.6);
\fill (v_1) circle (2pt) (v_2) circle (2pt) (v_r) circle (2pt);
\fill (-0.1,-0.1) rectangle (0.1,0.1);
\end{tikzpicture}
\end{equation*}
Then we have $\zeta^{}_{M}(X; \bk)=\zeta^{}_{M}(\bk)$ from \cite[Example 2.1]{O17}, and we have
$\zeta^{}_{\wcS, M}(X; \bk)=\zeta^{\sh}_{\wcS, M}(\bk)$
from \cite[Example 2.2]{OSY20}. 
\end{example}
\begin{remark}
We only consider the non-planar rooted tree which have no ordering of edges for each vertex. For example, we do not distinguish the following two $2$-colored rooted trees.
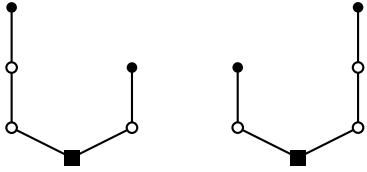
\begin{equation*}
\begin{tikzpicture}[thick]
\coordinate (v_1) at (-0.8,1.0) node at (v_1) [left]{};
\coordinate (v_2) at (-0.8,0.2) node at (v_2) [left]{};
\coordinate (v_3) at (-0.8,-0.6) node at (v_3) [left]{};
\coordinate (v_4) at (0.8,0.2) node at (v_4) [left]{};
\coordinate (v_5) at (0.8,-0.6) node at (v_5) [left]{};
\coordinate (rt) at (0,-1.0) node at (rt) [left]{};
\draw  (-0.8,1.0) -- node [right]{} (-0.8,0.2);
\draw  (-0.8,0.2) -- node [right]{} (-0.8,-0.6);
\draw  (-0.8,-0.6) -- node [right]{} (0.0,-1.0);
\draw  (0.8,0.2) -- node [right]{} (0.8,-0.6);
\draw  (0.8,-0.6) -- node [right]{} (0.0,-1.0);
\fill (v_1) circle (2pt) (v_4) circle (2pt);
\fill (-0.1,-1.1) rectangle (0.1,-0.9);
\filldraw[fill=white] (v_2) circle [radius=0.7mm];
\filldraw[fill=white] (v_3) circle [radius=0.7mm];
\filldraw[fill=white] (v_5) circle [radius=0.7mm];
\end{tikzpicture}
\qquad
\begin{tikzpicture}[thick]
\coordinate (w_1) at (-0.8,1.2) node at (w_1) [left]{};
\coordinate (w_2) at (-0.8,0.4) node at (w_2) [left]{};
\coordinate (w_3) at (0.8,2.0) node at (w_3) [left]{};
\coordinate (w_4) at (0.8,1.2) node at (w_4) [left]{};
\coordinate (w_5) at (0.8,0.4) node at (w_5) [left]{};
\coordinate (rt) at (0,0) node at (rt) [left]{};
\draw  (-0.8,1.2) -- node [right]{} (-0.8,0.4);
\draw  (-0.8,0.4) -- node [right]{} (0.0,0.0);
\draw  (0.8,2.0) -- node [right]{} (0.8,1.2);
\draw  (0.8,1.2) -- node [right]{} (0.8,0.4);
\draw  (0.8,0.4) -- node [right]{} (0.0,0.0);
\fill (w_1) circle (2pt) (w_3) circle (2pt);
\fill (-0.1,-0.1) rectangle (0.1,0.1);
\filldraw[fill=white] (w_2) circle [radius=0.7mm];
\filldraw[fill=white] (w_4) circle [radius=0.7mm];
\filldraw[fill=white] (w_5) circle [radius=0.7mm];
\end{tikzpicture}
\end{equation*}
Note that the two pairs $(X, \bk), (Y, \bl)$ consisting of the 2-colored rooted tree and an index coincide when $X=Y$ holds and each component of $\bk$ and $\bl$ coincides. In this case, of course we have $\zeta^{}_{M}(X; \bk)=\zeta^{}_{M}(Y; \bl)$ and $\zeta^{}_{\wcS, M}(X; \bk)=\zeta^{}_{\wcS, M}(Y; \bl)$
\end{remark}

We recall the definition of the essentially positivity.

\begin{definition}[{\cite[Definition 3.1]{OSY20}}]\label{def:ess_pos}
Let $X=(V, E, \rt, V_{\bullet})$ be a $2$-colored rooted tree. An index $\bk=(k_e)_{e \in E}$ is essentially positive when $k_{P(v, v')}=\sum_{e \in P(v, v')}k_e$ is positive for any two distinct  vertices $v, v' \in V_{\bullet}$.
\end{definition}
%

\subsection{Product structure on $2$-colored rooted trees}
Thanks to the non-planarity of rooted trees, we can consider the $\bQ$-vector space $2\CRT$ generated by the pair $(X, \bk)$ consisting of a $2$-colored rooted tree $X=(V, E, \rt, V_{\bullet})$ with $\rt \in V_{\bullet}$ and an essentially positive index $\bk$ on $X$. We introduce a product structure on $2\CRT$.

For a $2$-colored rooted tree $X_i=(V_i, E_i, \rt_i, V_{i, \bullet})$ and an index $\bk_i$ on $X_i \ (i=1,2)$, consider the new tree $(V, E)$ obtained by adjoining $\rt_1$ and $\rt_2$. That is, set
\begin{equation*}
V \coloneqq (V_1\setminus \{\rt_1\}) \sqcup (V_2\setminus\{\rt_2\}) \sqcup \{\rt\}, \qquad
E \coloneqq E_1 \sqcup E_2.
\end{equation*}
Here, $\rt$ is the new vertex corresponding $\rt_1$ (or $\rt_2$). We also set $V_\bullet \coloneqq (V_{1,\bullet} \setminus\{\rt_1\}) \sqcup (V_2\setminus\{\rt_2\}) \sqcup \{\rt\}$. Then we obtain a new 2-colored rooted tree $X \coloneqq (V, E, \rt, V_{\bullet})$ with $\rt \in V_\bullet$. Moreover, for an index $\bk_i=(k_{i, e})_{e \in E_i}$ on $X_i$, we define an index $\bk=(k_e)_{e\in E}$ on $X$ by
\begin{align*}
k_e
\coloneqq
\begin{cases}
k_{1, e} & \text {if $e \in E_1$}, \\
k_{2, e} & \text {if $e \in E_2$}.
\end{cases}
\end{align*}
Then, we obtain the $\bQ$-bilinear map $\circ \colon 2\CRT \times 2\CRT \rightarrow 2\CRT$. 

\begin{example}
Consider the following two pairs $(X, \bk)$ and $(X', \bk')$.
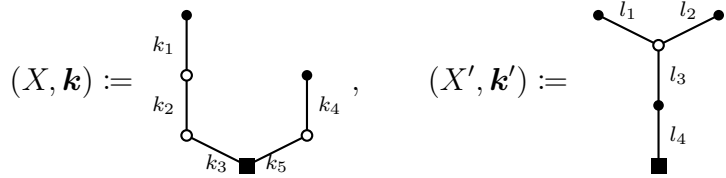
\begin{equation*}
(X, \bk)
\coloneqq
\begin{tikzpicture}[thick, baseline=0pt]
\coordinate (v_1) at (-0.8,1.0) node at (v_1) [left]{};
\coordinate (v_2) at (-0.8,0.2) node at (v_2) [left]{};
\coordinate (v_3) at (-0.8,-0.6) node at (v_3) [left]{};
\coordinate (v_4) at (0.8,0.2) node at (v_4) [left]{};
\coordinate (v_5) at (0.8,-0.6) node at (v_5) [left]{};
\coordinate (rt) at (0,-1.0) node at (rt) [left]{};
\draw  (-0.8,1.0) -- node [left]{\tiny $k_1$} (-0.8,0.2);
\draw  (-0.8,0.2) -- node [left]{\tiny $k_2$} (-0.8,-0.6);
\draw  (-0.8,-0.6) -- node [below=-1mm]{\tiny $k_3$} (0.0,-1.0);
\draw  (0.8,0.2) -- node [right]{\tiny $k_4$} (0.8,-0.6);
\draw  (0.8,-0.6) -- node [below=-1mm]{\tiny $k_5$} (0.0,-1.0);
\fill (v_1) circle (2pt) (v_4) circle (2pt);
\fill (-0.1,-1.1) rectangle (0.1,-0.9);
\filldraw[fill=white] (v_2) circle [radius=0.7mm];
\filldraw[fill=white] (v_3) circle [radius=0.7mm];
\filldraw[fill=white] (v_5) circle [radius=0.7mm];
\end{tikzpicture}, 
\qquad 
(X', \bk')
\coloneqq
\begin{tikzpicture}[thick, baseline=0pt]
\coordinate (w_1) at (-0.8,1.0) node at (w_1) [left]{};
\coordinate (w_2) at (0.8,1.0) node at (w_2) [left]{};
\coordinate (w_3) at (0.0,0.6) node at (w_3) [left]{};
\coordinate (w_4) at (0.0,-0.2) node at (w_4) [left]{};
\coordinate (rt) at (0.0,-1.0) node at (rt) [left]{};
\draw  (w_1) -- node [above]{\tiny $l_1$} (w_3);
\draw  (w_2) -- node [above]{\tiny $l_2$} (w_3);
\draw  (w_3) -- node [right]{\tiny $l_3$} (w_4);
\draw  (w_4) -- node [right]{\tiny $l_4$} (rt);
\fill (w_1) circle (2pt) (w_2) circle (2pt) (w_4) circle (2pt);
\fill (-0.1,-0.9) rectangle (0.1,-1.1);
\filldraw[fill=white] (w_3) circle [radius=0.7mm];
\end{tikzpicture}
\end{equation*}
Then we obtain the following.
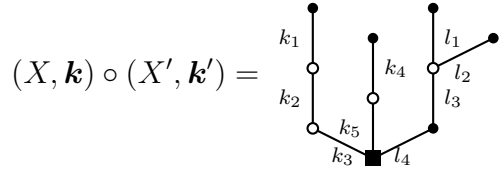
\begin{equation*}
(X, \bk) \circ (X', \bk')
=
\begin{tikzpicture}[thick, baseline=0pt]
\coordinate (rt) at (0,-1.0) node at (rt) [left]{};
\coordinate (v_1) at (-0.8,1.0) node at (v_1) [left]{};
\coordinate (v_2) at (-0.8,0.2) node at (v_2) [left]{};
\coordinate (v_3) at (-0.8,-0.6) node at (v_3) [left]{};
\coordinate (v_4) at (0.0,0.6) node at (v_4) [left]{};
\coordinate (v_5) at (0.0,-0.2) node at (v_5) [left]{};
\draw  (v_1) -- node [left]{\tiny $k_1$} (v_2);
\draw  (v_2) -- node [left]{\tiny $k_2$} (v_3);
\draw  (v_3) -- node [below=-1mm]{\tiny $k_3$} (rt);
\draw  (v_4) -- node [right]{\tiny $k_4$} (v_5);
\draw  (v_5) -- node [left]{\tiny $k_5$} (rt);
\fill (v_1) circle (2pt) (v_4) circle (2pt);
\fill (-0.1,-1.1) rectangle (0.1,-0.9);
\filldraw[fill=white] (v_2) circle [radius=0.7mm];
\filldraw[fill=white] (v_3) circle [radius=0.7mm];
\filldraw[fill=white] (v_5) circle [radius=0.7mm];
\coordinate (w_1) at (0.8,1.0) node at (w_1) [left]{};
\coordinate (w_2) at (1.6,0.6) node at (w_2) [left]{};
\coordinate (w_3) at (0.8,0.2) node at (w_3) [left]{};
\coordinate (w_4) at (0.8,-0.6) node at (w_4) [left]{};
\draw  (w_1) -- node [right]{\tiny $l_1$} (w_3);
\draw  (w_2) -- node [below]{\tiny $l_2$} (w_3);
\draw  (w_3) -- node [right]{\tiny $l_3$} (w_4);
\draw  (w_4) -- node [below=-1mm]{\tiny $l_4$} (rt);
\fill (w_1) circle (2pt) (w_2) circle (2pt) (w_4) circle (2pt);
\filldraw[fill=white] (w_3) circle [radius=0.7mm];
\end{tikzpicture}
\end{equation*}
\end{example}

It is easy to see that this $\bQ$-bilinear map $\circ$ makes $2\CRT$ an associative commutative $\bQ$-algebra.
The unit element is the 2-colored rooted tree $(\{\rt\}, \varnothing, \rt, \{\rt\})$ consisting of only one singleton $V=V_{\bullet}=\{\rt\}$.

\subsection{Product structure on harvestable pairs}
Next, we give another product structure on harvestable pairs. Firs, we recall the definition of the harvestability.
\begin{definition}[{\cite[Definition 3.14]{OSY20}}]\label{def:harv_pair}
Let $X=(V, E, \rt, V_{\bullet})$ be a $2$-colored rooted tree and $\bk=(k_e)_{e\in E}$ an index on $X$. The pair $(X, \bk)$ is harvestable if the following conditions hold:
\begin{description}
\item[(H1)] The root is a terminal of $(V, E)$. In particular, $\rt$ is in $V_{\bullet}$.
\item[(H2)] All elements of $V_{\circ}$ are branched points.
\item[(H3)] All elements of $V_{\bullet}$ are not branched points.
\item[(H4)] If $v \in V_{\circ}$ is the parent of $u \in V$, then $k_{\{v,u\}}$ is positive.
\item[(H5)] If $u, v \in V_{\bullet}$ and $\{u, v\} \in E$, then $k_{\{u, v\}}$ is positive.
\end{description}
Here, a branched point is a vertex of degree at least $3$, $\{u, v\}$ denotes the edge whose ends are $u$ and $v$, the parent of a vertex $v \neq \rt$ is the unique vertex $p$ satisfying $\{v, p\} \in P(\rt, v)$.
\end{definition}

In \cite[Proposition 3.18]{OSY20}, for a $2$-colored rooted tree $X=(V, E, \rt, V_{\bullet})$ with $\rt \in V_{\bullet}$ and an essentially positive index $\bk$, we constructed a harvestable pair $(X_{\h}, \bk_{\h})$, which we call the harvestable form of $(X, \bk)$, satisfying $\zeta^{}_{\wcS, M}(X; \bk)=\zeta^{}_{\wcS, M}(X_{\h}; \bk_{\h})$. The procedure obtaining $(X_{\h}, \bk_{\h})$ is as follows:
\begin{procedure}\label{proc:harv}
\begin{enumerate}
\item In $(X, \bk)$, if there exists an edge $e$ satisfying the condition in \cite[Proposition 3.4]{OSY20}, then contract $e$ according to \cite[Proposition 3.4]{OSY20}. Repeat this until we obtain a pair $(X_1, \bk_1)$ without such $e$.
\item In $(X_1, \bk_1)$, if there exists a pair of edges satisfying the condition in \cite[Proposition 3.5]{OSY20}, then joint them according to \cite[Proposition 3.5]{OSY20}. Repeat this until we obtain a pair $(X_2, \bk_2)$ without such edges.
\item In $(X_2, \bk_2)$, if there exists a black branched point $v \neq \rt$, then insert a new white vertex $v'$ together with an edge $\{v, v'\}$ at the location of $v$, and replace the edges $\{v, u\}$ by $\{v', u\}$ for vertices $u$ whose parent is $v$ in the original tree. Set the component of the index on the new edge $\{v, v'\}$ to be zero. Note that this is the inverse operation of the contraction according to \cite[Proposition 3.4]{OSY20}. Repeat this until we obtain a pair $(X_3, \bk_3)$ without such $v$.
\item In $(X_3, \bk_3)$, if the root is not terminal, then insert a new white vertex and an edge at location of the root, in the same way as (iii). The result is $(X_{\h}, \bk_{\h})$ we want to construct.
\end{enumerate}
\end{procedure} 
Let $2\CRT_{\h}$ be the $\bQ$-vector space generated by harvestable pairs. Then, from Procedure \ref{proc:harv}, we have a $\bQ$-linear map $h \colon 2\CRT \rightarrow 2\CRT_{\h}$ defined by sending $(X, \bk)$ to its harvestable form $(X_{\h}, \bk_{\h})$. Let $\circ_{\h} \colon 2\CRT \times 2\CRT \xrightarrow{\circ} 2\CRT \xrightarrow{h} 2\CRT_{\h}$ denote the composition of $\circ$ and $h$. 

\begin{theorem}
$(2\CRT_{\h}, \circ_{\h})$ is an associative commutative $\bQ$-algebra.
\end{theorem}
\begin{proof}
For $j=1, 2, 3$, let $(X_j, \bk_j)$ be the harvestable pairs of the following shape:
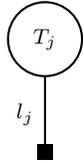
\begin{equation*}
\begin{tikzpicture}[thick]
\coordinate (T_j) at (0,1.5) node at (T_j) {\tiny $T_j$};
\coordinate (rt) at (0,0.0) node at (0,0) {};
\draw (0,1) -- node [left] {\tiny $l_j$} (rt);
\draw (T_j) circle [radius=14pt];
\fill (-0.1,-0.1) rectangle (0.1,0.1);
\end{tikzpicture}.
\end{equation*}
Note that since $(X_j, \bk_j)$ is harvestable, its root is terminal. Then, from the Procedure \ref{proc:harv} (ii) obtaining the harvestable pair, we see that $\bigl\{(X_1, \bk_1)\circ_{\h}(X_2, \bk_2)\bigr\} \circ_{\h} (X_3, \bk_3)$ coincide with the following harvestable pair:
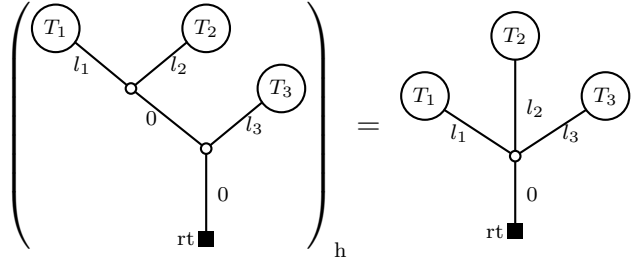
\begin{equation*}
\left(
\begin{tikzpicture}[thick,baseline=0pt]
\coordinate (T_1) at (-1.6,1.4) node at (T_1) {\tiny $T_1$};
\coordinate (T_2) at (0.4,1.4) node at (T_2) {\tiny $T_2$};
\coordinate (T_3) at (1.4,0.6) node at (T_3) {\tiny $T_3$};
\coordinate (u_1) at (-0.6,0.6) node at (u_1) {};
\coordinate (u_2) at (0.4,-0.2) node at (u_2) {};
\coordinate (rt) at (0.4,-1.4)  node[left] at (rt) {\tiny $\rt$};
\draw              (-1.35,1.22) -- node [left] {\tiny $l_1$} (u_1);
\draw              (0.15,1.22) -- node [right] {\tiny $l_2$} (u_1);
\draw              (u_1) -- node [left] {\tiny $0$} (u_2);
\draw              (1.15,0.42) -- node [right] {\tiny $l_3$} (u_2);
\draw              (u_2) -- node [right] {\tiny $0$} (rt);
\draw (T_1) circle [radius=9pt];
\draw (T_2) circle [radius=9pt];
\draw (T_3) circle [radius=9pt];
\fill (0.3,-1.5) rectangle (0.5,-1.3);
\filldraw[fill=white] (u_1) circle [radius=0.7mm];
\filldraw[fill=white] (u_2) circle [radius=0.7mm];
\end{tikzpicture}
\right)_{\h}
= \ 
\begin{tikzpicture}[thick,baseline=0pt]
\coordinate (T_1) at (-1.2,0.5) node at (T_1) {\tiny $T_1$};
\coordinate (T_2) at (0.0,1.3) node at (T_2) {\tiny $T_2$};
\coordinate (T_3) at (1.2,0.5) node at (T_3) {\tiny $T_3$};
\coordinate (u) at (0.0,-0.3) node at (u) {};
\coordinate (rt) at (0.0,-1.3)  node[left] at (rt) {\tiny $\rt$};
\draw              (-0.95,0.32) -- node [left] {\tiny $l_1$} (u);
\draw              (0.0,0.98) -- node [right] {\tiny $l_2$} (u);
\draw              (0.95,0.32) -- node [right] {\tiny $l_3$} (u);
\draw              (u) -- node [right] {\tiny $0$} (rt);
\draw (T_1) circle [radius=9pt];
\draw (T_2) circle [radius=9pt];
\draw (T_3) circle [radius=9pt];
\fill (-0.1,-1.4) rectangle (0.1,-1.2);
\filldraw[fill=white] (u) circle [radius=0.7mm];
\end{tikzpicture}
\end{equation*}
It is easy to see that $(X_1, \bk_1)\circ_{\h}\bigl\{(X_2, \bk_2) \circ_{\h} (X_3, \bk_3)\bigr\}$ coincides with the harvestable pair above by the similar procedure. This means that $\circ_{\h}$ is associative. The commutativity is also proved by the same way. 
\end{proof}
%

\section{Symmetrization map on $2\CRT$}

In this section, we define the $t$-adic symmetrization map $\wPhi$ on $2\CRT\jump{t}$, and prove Theorem \ref{mainthm:calu_phi}.

\begin{definition}
We define a $\bQ$-linear map $\wPhi \colon 2\CRT\jump{t} \rightarrow 2\CRT\jump{t}$ by 
\begin{equation*}
\wPhi(X, \bk)
\coloneqq \sum_{v \in V_{\bullet}}
(-1)^{k_{P(\rt, v)}}
\sum_{\bl=(l_e) \in \bZ^{P(\rt,v)}_{\ge0}}b_{P(\rt, v)}\binom{\bk}{\bl}(X_v, \bk+\bl)t^{l_{P(\rt, v)}}
\end{equation*}
for a $2$-colored rooted tree $X=(V, E, \rt, V_{\bullet})$ with $\rt \in V_{\bullet}$ and an essentially positive index $\bk=(k_e)_{e \in E}$ on $X$. Similarly, for the same $X$ and $\bk$ above, define a $\bQ$-linear map $\Phi \colon 2\CRT \rightarrow 2\CRT$ by  
\begin{equation*}
\Phi(X, \bk)
\coloneqq \sum_{v \in V_{\bullet}}(-1)^{k_{P(\rt, v)}}(X_v, \bk).
\end{equation*}
Note that $\Phi(X, \bk)$ is the constant term of $\wPhi(X, \bk)$.
\end{definition}
\begin{example}\label{ex:Phi_of_EZ3}
Let $(X, \bk)$ be the pair in Example \ref{ex:2crt_EZ} of $r=3$. Then we have
\begin{equation*}
\Phi(X, \bk)=
(-1)^{k_1+k_2+k_3}
\begin{tikzpicture}[thick,baseline=0pt]
\coordinate (v_4) at (0,0.9) node[left] at (v_4) {\tiny $v_4$};
\coordinate (v_3) at (0,0.3) node[left] at (v_3) {\tiny $v_3$};
\coordinate (v_2) at (0,-0.3) node[left] at (v_2) {\tiny $v_2$};
\coordinate (v_1) at (0,-0.9) node[left] at (v_1) {\tiny $v_1$};
\draw (v_4) -- node [right] {\tiny $k_3$} (v_3);
\draw (v_3) -- node [right] {\tiny $k_2$} (v_2);
\draw (v_2) -- node [right] {\tiny $k_1$} (v_1);
\fill  (v_4) circle (2pt) (v_3) circle (2pt) (v_2) circle (2pt);
\fill (-0.1,-1.0) rectangle (0.1,-0.8);
\end{tikzpicture}
+(-1)^{k_2+k_3}
\begin{tikzpicture}[thick,baseline=0pt]
\coordinate (v_1) at (-0.2,0.0) node[left] at (v_1) {\tiny $v_1$};
\coordinate (v_2) at (0,-0.9) node[left] at (v_2) {\tiny $v_2$};
\coordinate (v_3) at (0.2,0.0) node[right] at (v_3) {\tiny $v_3$};
\coordinate (v_4) at (0.4,0.9) node[right] at (v_4) {\tiny $v_4$};
\draw (v_4) -- node [right] {\tiny $k_3$} (v_3);
\draw (v_3) -- node [right] {\tiny $k_2$} (v_2);
\draw (v_2) -- node [left] {\tiny $k_1$} (v_1);
\fill  (v_1) circle (2pt) (v_3) circle (2pt) (v_4) circle (2pt);
\fill (-0.1,-1.0) rectangle (0.1,-0.8);
\end{tikzpicture}
+(-1)^{k_3}
\begin{tikzpicture}[thick,baseline=0pt]
\coordinate (v_1) at (-0.4,0.9) node[left] at (v_1) {\tiny $v_1$};
\coordinate (v_2) at (-0.2,0.0) node[left] at (v_2) {\tiny $v_2$};
\coordinate (v_3) at (0.0,-0.9) node[right] at (v_3) {\tiny $v_3$};
\coordinate (v_4) at (0.2,0.0) node[right] at (v_4) {\tiny $v_4$};
\draw (v_4) -- node [right] {\tiny $k_3$} (v_3);
\draw (v_3) -- node [left] {\tiny $k_2$} (v_2);
\draw (v_2) -- node [left] {\tiny $k_1$} (v_1);
\fill  (v_1) circle (2pt) (v_2) circle (2pt) (v_4) circle (2pt);
\fill (-0.1,-1.0) rectangle (0.1,-0.8);
\end{tikzpicture}
+
\begin{tikzpicture}[thick,baseline=0pt]
\coordinate (v_1) at (0,0.9) node[left] at (v_1) {\tiny $v_1$};
\coordinate (v_2) at (0,0.3) node[left] at (v_2) {\tiny $v_2$};
\coordinate (v_3) at (0,-0.3) node[left] at (v_3) {\tiny $v_3$};
\coordinate (v_4) at (0,-0.9) node[left] at (v_4) {\tiny $v_4$};
\draw (v_1) -- node [right] {\tiny $k_1$} (v_2);
\draw (v_2) -- node [right] {\tiny $k_2$} (v_3);
\draw (v_3) -- node [right] {\tiny $k_3$} (v_4);
\fill  (v_1) circle (2pt) (v_2) circle (2pt) (v_3) circle (2pt);
\fill (-0.1,-1.0) rectangle (0.1,-0.8);
\end{tikzpicture}
\end{equation*}
\end{example}
\begin{definition}\label{def:sh_prod}
We define the shuffle product $\sh$ on $\frH$ as the $\bQ$-bilinear map $\sh \colon \frH \times \frH \rightarrow \frH$ defined by the following rules.
\begin{enumerate}
\item $a \sh 1=1\sh a=1$ for any $a \in \frH$.
\item $(a_1u_1)\sh(a_2u_2)=(a_1\sh a_2u_2)u_1+(a_1u_1 \sh a_2)u_2$ for any $a_1, a_2 \in \frH$ and $u_1, u_2 \in \{x, y\}$.
\end{enumerate}
\end{definition}

From \cite[Definition 3.16]{OSY20}, for a harvestable pair $(X, \bk)$, we constructed an element $w(X, \bk) \in \frH^1$ inductively as follows : 
\begin{procedure}\label{proc:w}
\begin{enumerate}
\item For $(X, \bk)$ in Example \ref{ex:2crt_EZ}, we define $w(X, \bk) \coloneqq z_{k_1}\cdots z_{k_r}$.
\item Let $(X, \bk)$ (resp. $(X_j, \bk_j)$) be given by the left hand side (resp. the right hand side) of the following diagrams.
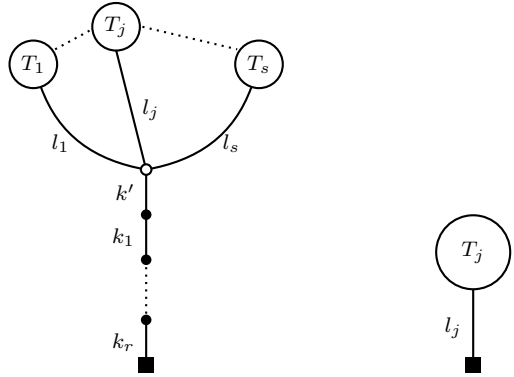
\begin{equation*}
\begin{tikzpicture}[thick]
\coordinate (T_1) at (-1.5,4) node at (T_1) {\tiny $T_1$};
\coordinate (T_j) at (-0.4,4.5) node at (T_j) {\tiny $T_j$};
\coordinate (T_s) at (1.5,4) node at (T_s) {\tiny $T_s$};
\coordinate (v_0) at (0,2.6) node at (0,2.6) {};
\coordinate (v_1) at (0,2.0) node at (0,2.0) {};
\coordinate (v_2) at (0,1.4)  node at (0,1.4) {};
\coordinate (v_r) at (0,0.6) node at (0,0.6) {};
\coordinate (v_{r+1}) at (0,0.0) node at (0,0) {};
\draw              (-1.4,3.7) to [out=290,in=170] node [left] {\tiny $l_1$} (v_0);
\draw [dotted] (-1.21,4.2) -- (-0.75,4.45);
\draw              (-0.4,4.19) -- node [right] {\tiny $l_j$} (v_0);
\draw [dotted] (-0.05,4.5) -- (1.21,4.2);
\draw              (1.4,3.7) to [out=250, in=10] node [right] {\tiny $l_s$} (v_0);
\draw              (v_0) -- node [left] {\tiny $k'$} (v_1);
\draw              (v_1) -- node [left] {\tiny $k_1$} (v_2);
\draw [dotted] (v_2) -- (v_r);
\draw              (v_r) -- node [left] {\tiny $k_r$} (v_{r+1});
\draw (T_1) circle [radius=9pt];
\draw (T_j) circle [radius=9pt];
\draw (T_s) circle [radius=9pt];
\fill (v_1) circle (2pt) (v_2) circle (2pt) (v_r) circle (2pt);
\fill (-0.1,-0.1) rectangle (0.1,0.1);
\filldraw[fill=white] (v_0) circle [radius=0.7mm];
\end{tikzpicture}
\hspace{2cm}
\begin{tikzpicture}[thick]
\coordinate (T_j) at (0,1.5) node at (T_j) {\tiny $T_j$};
\coordinate (rt) at (0,0.0) node at (0,0) {};
\draw (0,1) -- node [left] {\tiny $l_j$} (rt);
\draw (T_j) circle [radius=14pt];
\fill (-0.1,-0.1) rectangle (0.1,0.1);
\end{tikzpicture}
\end{equation*}
Assume that $w_j=w(X_j, \bk_j)$ for $j=1, \ldots, s$ are already defined. Then we define 
\begin{equation*}
w(X, \bk) \coloneqq (w_1 \sh \cdots \sh w_s)x^{k'}z_{k_1}\cdots z_{k_r}.
\end{equation*}
\end{enumerate}
\end{procedure}
Then, we have $\zeta^{}_{\wcS, M}(X; \bk)=Z^{\sh}_{\wcS, M}\bigl(w(X, \bk)\bigr)$ from \cite[Theorem 3.17]{OSY20}, and $\zeta^{}_{M}(X; \bk)=Z_{M}\bigl(w(X, \bk)\bigr)$ from \cite[Proposition 3.2]{O17}. Here, $Z^{\sh}_{\wcS, M} \colon \frH^1 \rightarrow \bQ\jump{t}$ (resp.~$Z_{M} \colon \frH^1 \rightarrow \bQ$) is the $\bQ$-linear map defined by $Z^{\sh}_{\wcS, M}(z_{\bk})=\zeta^{\sh}_{\wcS, M}(\bk)$ (resp.~$Z_{M}(z_{\bk}) =\zeta^{}_{M}(\bk)$) for any index $\bk$. Note that $Z^{\sh}_{\wcS, M}=Z_{M} \circ \wphi$ by definition. 

\begin{example}\label{ex:hybrid}
For positive integers $k_1, k_2, k_3, k_4$ and a non-negative integer $l$, consider the following harvestable pair $(X, \bk)$.
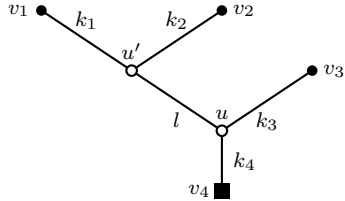
\begin{equation*}
\begin{tikzpicture}[thick,baseline=-2pt]
\coordinate (v_1) at (-2.0,1.2) node[left] at (v_1) {\tiny $v_1$};
\coordinate (v_2) at (0.4,1.2) node[right] at (v_2) {\tiny $v_2$};
\coordinate (v_3) at (1.6,0.4) node[right] at (v_3) {\tiny $v_3$};
\coordinate (u') at (-0.8,0.4) node[above] at (u') {\tiny $u'$};
\coordinate (u) at (0.4,-0.4) node[above] at (u) {\tiny $u$};
\coordinate (rt) at (0.4,-1.2) node[left] at (rt) {\tiny $v_4$};
\draw (v_1) -- node [above] {\tiny $k_1$} (u');
\draw (v_2) -- node [above] {\tiny $k_2$} (u');
\draw (u') -- node [below] {\tiny $l$} (u);
\draw (v_3) -- node [below] {\tiny $k_3$} (u);
\draw (u) -- node [right] {\tiny $k_4$} (rt);
\fill (v_1) circle (2pt) (v_2) circle (2pt) (v_3) circle (2pt);
\fill (0.5,-1.3) rectangle (0.3,-1.1);
\filldraw[fill=white] (u) circle [radius=0.7mm];
\filldraw[fill=white] (u') circle [radius=0.7mm];
\end{tikzpicture}
\end{equation*}
By Procedure \ref{proc:w} (i), we have 
\begin{equation*}
w_i
\coloneqq
w\biggl(
\begin{tikzpicture}[thick,baseline=-2pt]
\coordinate (v_i) at (0.0,0.5) node[left] at (v_i) {\tiny $v_i$};
\coordinate (rt) at (0.0,-0.5) node at (rt) {};
\draw (v_i) -- node [right] {\tiny $k_i$} (rt);
\fill (v_i) circle (2pt);
\fill (-0.1,-0.6) rectangle (0.1,-0.4);
\end{tikzpicture}
\biggr)
=z_{k_i}
\end{equation*}
for $i=1,2,3$. Thus, by use of Procedure \ref{proc:w} (ii) twice, we have
\begin{equation*}
w(X, \bk)=\{(w_1 \sh w_2)x^{l}\sh w_3\}x^{k_4}=\{(z_{k_1} \sh z_{k_2})x^{l}\sh z_{k_3}\}x^{k_4}.
\end{equation*}
\end{example}
From Procedure \ref{proc:w}, we can regard that $w$ is the $\bQ$-linear map $w \colon 2\CRT_{\h} \rightarrow \frH^1$ which sends a harvestable pair $(X, \bk)$ to $w(X, \bk)$. We use the same symbol $w$ to the $\bQ$-linear extension $2\CRT_{\h}\jump{t} \rightarrow \frH^1\jump{t}$ of $w$.

Now, we recall the statement of our main theorem.

\begin{theorem}{(=Theorem \ref{mainthm:calu_phi})}\label{thm:comparison_phi}
For a $2$-colored rooted tree $X=(V, E, \rt, V_{\bullet})$ with $\rt \in V_{\bullet}$ and an essentially positive index $\bk=(k_e)_{e \in E}$ on $X$, we have
\begin{align*}
\wphi\bigl(w(X_{\h}, \bk_{\h})\bigr)
=\sum_{v \in V_{\bullet}}(-1)^{k_{P(\rt, v)}}\sum_{\bl \in \bZ^{P(\rt, v)}_{\ge0}}
b_{P(\rt, v)}\binom{\bk}{\bl}w\bigl(X_{v, \h}, (\bk+\bl)_{\h}\bigr)t^{l_{P(\rt, v)}}.
\end{align*}
\end{theorem}
\begin{remark}\label{rm:comm_diag}
Let $h \colon 2\CRT\jump{t} \rightarrow 2\CRT_{\h}\jump{t}$ be the $\bQ\jump{t}$-linear extension. Then Theorem \ref{mainthm:calu_phi} means that the following diagram commutes. 
\begin{equation*}
\begin{tikzcd}
  2\CRT\jump{t} \ar[r, "h"] \arrow[d, "\wPhi"'] & 2\CRT_{\h}\jump{t} \ar[r, "w"] & \frH^1\jump{t} \ar[d, "\wphi"] \\
  2\CRT\jump{t} \ar[r, "h"] & 2\CRT_{\h}\jump{t} \ar[r, "w"'] & \frH^1\jump{t}
\end{tikzcd}
\end{equation*}
\end{remark}

For the proof of Theorem \ref{mainthm:calu_phi}, we need the following proposition for the multiple harmonic sums,  which is proved by the similar argument of \cite[Proposition 3.6]{OSY20}. 

\begin{proposition}\label{prop:root_change_zeta}
For a harvestable pair $(X, \bk)$, we have
\begin{align*}
\zeta^{\sh}_{\wcS, M}(X; \bk)
=\sum_{v \in V_{\bullet}}(-1)^{k_{P(\rt, v)}}\sum_{\bl \in \bZ^{P(\rt, v)}_{\ge0}}
b_{P(\rt, v)}\binom{\bk}{\bl}\zeta^{}_{M}\bigl(X_{v, \h}; (\bk+\bl)_{\h}\bigr)t^{l_{P(\rt, v)}}.
\end{align*}
\end{proposition}
\begin{proof}
For $e \in E, u \in V_{\bullet}$ and $(m_v)_{v \in V_{\bullet}} \in I_M(V_{\bullet}, u)$, set
\begin{align*}
L_e(X, u; (m_v)_{v \in V_{\bullet}})
\coloneqq
\sum_{v \in V_{\bullet} \text{ s.t. }e \in P(\rt, v)}(m_v+\delta_{u,v}t).
\end{align*}
Suppose $u \neq \rt$. If $e \in P(\rt, u)$, we have 
\begin{equation}\label{eq:e_in_path}
L_e\bigl(X, u; (m_v)_{v \in V_{\bullet}}\bigr)
=t-L_e\bigl(X_u, u; (m_v)_{v \in V_{\bullet}}\bigr)
=t-\sum_{v \in V_{\bullet} \text{ s.t. } e \in P(u, v)}m_v,
\end{equation}
and if $e \not \in P(\rt, v)$, we have
\begin{equation}\label{eq:e_notin_path}
L_e\bigl(X, u;(m_v)_{v \in V_{\bullet}}\bigr)
=L_e\bigl(X_u, u;(m_v)_{v \in V_{\bullet}}\bigr)
=\sum_{v \in V_{\bullet} \text{ s.t. } e \in P(u, v)}m_v.
\end{equation}
By \eqref{eq:e_in_path} and \eqref{eq:e_notin_path}, we have
\begin{align*}
\zeta^{}_{M}(X, u; \bk)
&=\sum_{\substack{(m_v)_{v \in V_{\bullet}\setminus\{u\}} \in \bZ^{V_{\bullet} \setminus\{u\}}_{\ge1} \\ \sum_{v \neq u}m_v<M}}
\prod_{e \in P(u,v)}(t-L)^{-k_e}\prod_{e \not\in P(u, v)}L^{-k_e},
\end{align*}
where $L \coloneqq L_e(X_u, u; (m_v)_{v \in V_{\bullet}})$. By using
\begin{equation*}
(t-L)^{-k_e}=(-1)^{k_e}\sum_{l_e \ge0}\binom{k_e+l_e-1}{l_e}L^{-k_e-l_e}t^{l_e},
\end{equation*}
we obtain
\begin{equation}\label{eq:conclusion}
\begin{split}
\zeta^{}_{M}(X, u; \bk)
&=(-1)^{k_{P(\rt,u)}} \sum_{\bl=(l_e)\in \bZ^{P(\rt, u)}_{\ge0}}b_{P(\rt, u)}\binom{\bk}{\bl}t^{l_{P(\rt,u)}}\\
&\qquad\qquad\qquad\qquad \times\sum_{\substack{(m_v)_{v \in V_{\bullet}\setminus\{u\}}  \in \bZ^{V_{\bullet} \setminus\{u\}}_{\ge1} \\ \sum_{v \neq u}m_v<M}}
\prod_{e \in E}L^{-k_e-l_e}.
\end{split}
\end{equation}
By definition, the inner sum in \eqref{eq:conclusion} coincides with $\zeta^{}_{M}(X_u; \bk+\bl)$,
and then coincides with $\zeta^{}_{M}\bigl(X_{u, \h}; (\bk+\bl)_{\h}\bigr)$ from \cite[Proposition 2.9]{O17}.
Summing up \eqref{eq:conclusion} for all $u \in V_{\bullet}$, we obtain the desired formula.
\end{proof}

We use the same symbol $Z_{M}$ to write the extension of the map $Z_{M}$ to the $\bQ$-linear map $\frH^1\jump{t} \rightarrow \bQ\jump{t}$. Let $\ast \colon \frH^1 \times \frH^1 \rightarrow \frH^1$ be the harmonic product on $\frH^1$. It is known \cite[Theorem 2.1]{H97} that $\frH^1_{\ast} \coloneqq (\frH^1, \ast)$ is an associative commutative $\bQ$-algebra. The following theorem immediately follows from \cite[Theorem 3.1]{Y13}.

\begin{theorem}\label{thm:extYamamoto}
For $W\in \frH^1\jump{t}$, let $Z(W)$ denote the element $\bigl\{Z_M(W)\bigr\}_{M \ge0}$ of $(\bQ\jump{t})^{\bZ_{\ge0}}$. Then the resulting $\bQ\jump{t}$-algebra homomorphism $Z \colon \frH^1_{\ast}\jump{t} \rightarrow (\bQ\jump{t})^{\bZ_{\ge0}}$ is injective.
\end{theorem}
\begin{proof}
For $W=\sum_{n\ge0}w_nt^n \in \frH^1_{*}\jump{t}$, suppose that $Z_M(W)=\sum_{n \ge0}Z_{M}(w_n)=0$ for all $M \ge0$. Then, for each $n$, we have $Z_M(w_n)=0$ for all $M\ge0$. Thus, from \cite[Theorem 3.1]{Y13}, we see that $w_n=0$ for all $n \ge0$, which gives $W=0$.
\end{proof}
\begin{proof}[Proof of Theorem \ref{thm:comparison_phi}]
From \cite[Proposition 3.18]{OSY20}, we have $\zeta^{}_{\wcS,M}(X;\bk)=\zeta^{}_{\wcS, M}(X_{\h}, \bk_{\h})=Z^{\sh}_{\wcS, M}\bigl(w\bigl(X_{\h}, \bk_{\h}\bigr)\bigr)=Z_{M}\circ \wphi\bigl(w\bigl(X_{\h}, \bk_{\h}\bigr)\bigr)$. Since Proposition \ref{prop:root_change_zeta} holds for all $M\ge0$ and any pair $(X, \bk)$ in $2\CRT$, the statement follows immediately from Theorem \ref{thm:extYamamoto}.
\end{proof}
\begin{remark}\label{rm:t-adicKaneko}
By using Theorem \ref{thm:extYamamoto} and \cite[Theorem 3.10]{OSY20}, we obtain the following equality in $\frH^1\jump{t}$, which gives the shuffle relation for $\wcS$-MZV (\cite{J21}.~cf.~\cite[Corollary 3.11]{OSY20}). That is, for indices $\bk, \bl$, we have
\begin{equation*}\label{eq:sh_frH^1t}
\wphi(z_{\bk} \sh z_{\bl})
=(-1)^{\wt(\bl)}
\sum_{\bl' \in \bZ^{\dep(\bl)}_{\ge0}}b\binom{\bl}{\bl'}\wphi(z_{\bk}z_{\overline{\bl+\bl'}})t^{\wt(\bl')}.
\end{equation*}
In particular, we have
\begin{equation*}\label{eq:sh_frH^1}
\phi(z_{\bk}\sh z_{\bl})=(-1)^{\wt(\bl)}\phi(z_{\bk} z_{\overline{\bl}}),
\end{equation*}
which is equivalent to \cite[Proposition 9.7]{Kan19}. We obtain the analogous identity of \cite[Proposition 3.4]{Kam16} in $\frH^1$ by using \cite[Theorem 3.1]{Y13}.
\end{remark}
%

Now we obtain a $t$-adic generalization of Theorem \ref{thm:BTT} (\cite[Theorem 4.6]{BTT21}).
\begin{corollary}\label{cor:t-adicBTT}
For positive integers $k_1, \ldots, k_{r+1}$, we have
\begin{align*}
\wphi\bigl((z_{k_1}\sh \cdots \sh z_{k_r})x^{k_{r+1}}\bigr)
&=(z_{k_1}\sh \cdots \sh z_{k_r})x^{k_{r+1}}\\
&\quad +\sum_{i=1}^{r}(-1)^{k_i+k_{r+1}}\sum_{l, l' \ge0}\binom{k_i+l-1}{l}\binom{k_{r+1}+l'-1}{l'}\\
&\qquad \qquad \times(z_{k_1}\sh \cdots \sh \check{z}_{k_{i}} \sh \cdots \sh z_{k_r} \sh z_{k_{r+1}+l'})x^{k_i+l}t^{l+l'}.
\end{align*}
Recall that the symbol $\check{z}_{k_i}$ means that the factor $z_{k_i}$ is skipped. 
\end{corollary}
\begin{proof}
Let $(X, \bk)$ be the pair illustrated as follows.
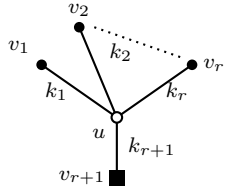
\begin{equation*}
\begin{tikzpicture}[thick]
\coordinate (v_1) at (-1,1.5) node at (v_1) [above left] {\tiny $v_1$};
\coordinate (v_2) at (-0.5,2.0) node at (v_2) [above] {\tiny $v_2$};
\coordinate (v_r) at (1,1.5) node at (v_r) [right] {\tiny $v_r$};
\coordinate (w) at (0,0.8) node at (w) [below left] {\tiny $u$};
\coordinate (v_{r+1}) at (0,0) node at (v_{r+1}) [left] {\tiny $v_{r+1}$};
\draw (v_1) -- node [left] {\tiny $k_1$} (w);
\draw (v_2) -- node [above right] {\tiny $k_2$} (w);
\draw (v_r) -- node [right] {\tiny $k_r$} (w);
\draw (w) -- node [right] {\tiny $k_{r+1}$} (v_{r+1});
\draw [dotted] (-0.3,2.0) -- (0.85,1.6);
\fill (v_1) circle (2pt) (v_2) circle (2pt) (v_r) circle (2pt);
\fill (-0.1,-0.1) rectangle (0.1,0.1);
\filldraw[fill=white] (w) circle[radius=0.7mm];
\end{tikzpicture}
\end{equation*}
For non-negative integers $l, l'$ and $1\le i \le r$, set $\bk_i(l, l') \coloneqq (k_1, \ldots, k_{i-1}, k_{i+1}, \ldots, k_r, k_{r+1}+l', k_i+l)$. Then we have
\begin{align*}
\wPhi(X, \bk)=(X_{v_{r+1}}, \bk)
+\sum_{i=1}^{r}\sum_{l, l' \ge0}\binom{k_i+l-1}{l}\binom{k_{r+1}+l'-1}{l'}
(X_{v_i}, \bk_i(l, l'))t^{l+l'}.
\end{align*}
Note that $(X, \bk)=(X_{v_{r+1}}, \bk)$. Since $k_1, \ldots, k_{r+1}$ are positive, $(X, \bk)$ and $(X_{v_i}, \bk_i(l, l')) \ (1\le i \le r, l, l' \ge0)$ are harvestable. Thus by Procedure \ref{proc:w}, we have
\begin{align*}
w(X, \bk)
&=(z_{k_1} \sh \cdots \sh z_{k_r})x^{k_{r+1}},\\
w(X_{v_i}, \bk_i(l, l'))
&=(z_{k_1}\sh \cdots \sh \check{z}_{k_{i}} \sh \cdots \sh z_{k_r} \sh z_{k_{r+1}+l'})x^{k_i+l}.
\end{align*}
Therefore, by Theorem \ref{mainthm:calu_phi}, we obtain the desired formula.
\end{proof}

The next proposition shows that there exists a $2$-colored rooted tree whose image by the symmetrization map $\Phi$ vanishes.

\begin{proposition}\label{prop:sym_tree}
Let $X=(V, E, \rt, V_{\bullet})$ be a $2$-colored rooted tree with $\rt \in V_{\bullet}$. Assume the following conditions.
\begin{enumerate}
\item There exist subsets $V_i \subset V, E_i \subset E, V_{i, \bullet} \subset V_{\bullet}\ (i=1,2)$ and an edge $f \in E$ satisfying
\begin{equation*}
V=V_1 \sqcup V_2, \quad E=E_1\sqcup E_2 \sqcup \{f\}, \quad V_{\bullet}=V_{1, \bullet} \sqcup V_{2, \bullet}.
\end{equation*}
\item $\rt \in V_{1, \bullet}$ and there exists an isomorphism $F \colon (V_1, E_1, \rt) \rightarrow (V_2, E_2, F(\rt))$ of rooted trees satisfying that for $v \in V_{\bullet}$, $v$ is an element in $V_{1, \bullet}$ if and only if $F(v)$ is an element in $V_{2, \bullet}$. 
\item For an essentially positive index $\bk=(k_e)_{e\in E} \in \bZ^{E}_{\ge0}$ on $X$, $k_f$ is odd and $k_{\{a, b\}}=k_{\{F(a), F(b)\}}$ for any $e=\{a, b\} \in E$.
\end{enumerate}
Then, we have
\begin{equation*}
\Phi(X, \bk)=0.
\end{equation*}
In particular, we have
\begin{equation} \label{eq:phi_vanish}
\phi\bigl(w(X_{\h}, \bk_{\h})\bigr)=0.
\end{equation}
\end{proposition}
\begin{remark}\label{rm:from OSY}
\eqref{eq:phi_vanish} follows immediately from Theorem \ref{thm:extYamamoto} and \cite[Proposition 3.6]{OSY20}.
\end{remark}
\begin{proof}
Write $f=\{u, F(u)\}$. Then we can illustrate $(X, \bk)$ as follows.
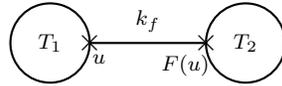
\begin{equation*}
\begin{tikzpicture}[thick,baseline=0pt]
\coordinate (T_1) at (-1.3,0.0) node at (T_1) {\tiny $T_1$};
\coordinate (T_2) at (1.3,0.0) node at (T_2) {\tiny $T_2$};
\coordinate (u) at (-0.65,0.0) node[below] at (u) {\tiny $u$};
\coordinate (u') at (0.5,0.0) node[below] at (u') {\tiny $F(u)$};
\draw (-0.775,0.0) -- node [above] {\tiny $k_{f}$} (0.775,0.0);
\draw (-0.775,0.0) -- node {$\times$} (-0.775,0.0);
\draw (0.775,0.0) -- node {$\times$} (0.775,0.0);
\draw (T_1) circle [radius=15pt];
\draw (T_2) circle [radius=15pt];
\end{tikzpicture}
\end{equation*}
Here, we use the symbol $\times$ to indicate a vertex which may or may not belong to $V_\bullet$ and 
may or may not be the root, and we set $T_i=(V_i, E_i) \ (i=1,2)$. Since $k_{P(u, v)}=k_{P(F(u), F(v))}$ for $v \in V_{1,\bullet}$ from the assumption on $\bk$, we have 
\begin{align*}
k_{P(\rt, F(v))}
&=k_{P(\rt, u)}+k_f+k_{P(F(u), F(v))}\\
&=k_{P(\rt, u)}+k_f+k_{P(u,v)}=k_{P(\rt, v)}+k_f.
\end{align*}
Since $k_f$ is odd, we have
\begin{equation*}
(-1)^{k_{P(\rt, F(v))}}=-(-1)^{k_{P(\rt, v)}}
\end{equation*}
for any $v \in V_{1, \bullet}$. Moreover, from the assumption on $(X, \bk)$ and the non-planarity of rooted trees, we have $(X_v, \bk)=(X_{F(v)}, \bk)$ for any $v \in V_{1, \bullet}$.
Thus we obtain
\begin{align*}
\Phi(X, \bk)
&=\sum_{v \in V_{1, \bullet}}(-1)^{k_{P(\rt, v)}}(X_v, \bk)
+\sum_{v \in V_{2, \bullet}}(-1)^{k_{P(\rt, v)}}(X_v, \bk)\\
&=\sum_{v \in V_{1, \bullet}}(-1)^{k_{P(\rt, v)}}(X_v, \bk)
+\sum_{v \in V_{1, \bullet}}(-1)^{k_{P(\rt, F(v))}}(X_{F(v)}, \bk)\\
&=\sum_{v \in V_{1, \bullet}}(-1)^{k_{P(\rt, v)}}(X_v, \bk)
-\sum_{v \in V_{1, \bullet}}(-1)^{k_{P(\rt, v)}}(X_v, \bk)=0,
\end{align*}
which completes the proof.
\end{proof}

\begin{example}
Consider the harvestable pair $(X, \bk)$ in Example \ref{ex:hybrid}.
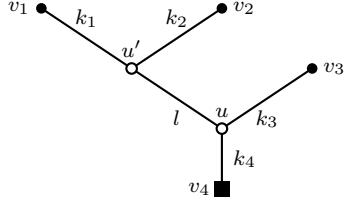
\begin{equation*}
\begin{tikzpicture}[thick,baseline=-2pt]
\coordinate (v_1) at (-2.0,1.2) node[left] at (v_1) {\tiny $v_1$};
\coordinate (v_2) at (0.4,1.2) node[right] at (v_2) {\tiny $v_2$};
\coordinate (v_3) at (1.6,0.4) node[right] at (v_3) {\tiny $v_3$};
\coordinate (u') at (-0.8,0.4) node[above] at (u') {\tiny $u'$};
\coordinate (u) at (0.4,-0.4) node[above] at (u) {\tiny $u$};
\coordinate (rt) at (0.4,-1.2) node[left] at (rt) {\tiny $v_4$};
\draw (v_1) -- node [above] {\tiny $k_1$} (u');
\draw (v_2) -- node [above] {\tiny $k_2$} (u');
\draw (u') -- node [below] {\tiny $l$} (u);
\draw (v_3) -- node [below] {\tiny $k_3$} (u);
\draw (u) -- node [right] {\tiny $k_4$} (rt);
\fill (v_1) circle (2pt) (v_2) circle (2pt) (v_3) circle (2pt);
\fill (0.5,-1.3) rectangle (0.3,-1.1);
\filldraw[fill=white] (u) circle [radius=0.7mm];
\filldraw[fill=white] (u') circle [radius=0.7mm];
\end{tikzpicture}
\end{equation*}
Here we set $k_i=k_{e_i}$ and $l=k_f$.
Set $V_{1, \bullet} \coloneqq \{v_3, v_4\}, =V_{2, \bullet} \coloneqq \{v_1, v_2\}$, $V_1 \coloneqq \{u\} \sqcup V_{1,\bullet}, V_2 \coloneqq \{u'\}\sqcup V_{2, \bullet}$, 
and $E_1 \coloneqq \{e_3, e_4\}, E_2 \coloneqq \{e_1, e_2\}$.
Then we see that if $k_1=k_4, k_2=k_3$ and $l$ is odd, the pair $(X, \bk)$ satisfies the assumptions of Proposition \ref{prop:sym_tree}. Indeed, define the map of trees $F \colon (V_1, E_1) \rightarrow (V_2, E_2)$ by $F(v_3)=v_2, F(v_4)=v_1$ and $F(u)=u'$. Then $F$ is an isomorphism $(V_1, E_1, v_4) \rightarrow (V_2, E_2, v_1)$ of rooted trees. Therefore, by Proposition \ref{prop:sym_tree}, we have $\Phi(X, \bk)=0$, in particular, $\phi\bigl(w(X, \bk)\bigr)=0$. 

On the other hand, $\phi\bigl(w(X, \bk)\bigr)=0$ is obtained by the explicit calculation as follows.
By Theorem \ref{mainthm:calu_phi} and Example \ref{ex:hybrid}, we have
\begin{align*}
&\phi\bigl(w(X, \bk)\bigr)=\phi\bigl(\{(z_{k_1}\sh z_{k_2})x^l\sh z_{k_3} \}x^{k_4}\bigr)\\
&=(-1)^{k_1+l+k_4}\{(z_{k_3}\sh z_{k_4})x^l\sh z_{k_2}\}x^{k_1}
+(-1)^{k_2+l+k_4}\{(z_{k_3}\sh z_{k_4})x^l\sh z_{k_1} \}x^{k_2}\\
&\quad +(-1)^{k_3+k_4}\{(z_{k_1}\sh z_{k_2})x^l\sh z_{k_4}\}x^{k_3}
+\{(z_{k_1}\sh z_{k_2})x^l\sh z_{k_3} \}x^{k_4}.
\end{align*}
Thus, if $k_4=k_1, k_2=k_3$ and $l$ is odd, we obtain $\phi\bigl(w(X, \bk)\bigr)=0$.
\end{example}

\begin{remark}
Remark \ref{rm:t-adicKaneko}, Corollary \ref{cor:t-adicBTT} and Proposition \ref{prop:sym_tree} give some non-trivial elements in $\Ker\wPhi$ or $\Ker\wphi$. It seems interesting to study the explicit description of $\Ker\wPhi$ and $\Ker\wphi$ (or $\Ker \Phi$ and $\Ker \phi$).
\end{remark}


\end{document}